\documentclass[a4paper,10pt,leqno, english]{amsart}
\title[The topological K-theory of crystallographic groups.]
{The topological K-theory of crystallographic groups with holonomy $\mathbb{Z}/2$}

\author{Mario Vel\'asquez}
\address{ Departamento de Matem\'aticas
\\Pontificia Universidad Javeriana\\Cra. 7 No. 43-82 - Edificio Carlos Ortíz 6to piso\\ Bogot\'a D.C, Colombia}
\email{mavelasquezm@gmail.com}
\urladdr{https://sites.google.com/site/mavelasquezm/}
         \date{\today}


\usepackage{multicol}
\usepackage{hyperref}
\usepackage{pdfsync}
\usepackage{calc}
\usepackage{enumerate,amssymb}
\usepackage[arrow,curve,matrix,tips,2cell]{xy}
  \SelectTips{eu}{10} \UseTips
  \UseAllTwocells

\usepackage{pb-diagram}
\usepackage{verbatim}

\DeclareMathAlphabet\EuR{U}{eur}{m}{n}
\SetMathAlphabet\EuR{bold}{U}{eur}{b}{n}

\makeindex             

\theoremstyle{plain}
\newtheorem{theorem}{Theorem}[section]

\newtheorem{proposition}[theorem]{Proposition}

\theoremstyle{definition}

\newtheorem{remark}[theorem]{Remark}

{\catcode`@=11\global\let\c@equation=\c@theorem}

\newcommand{\comsquare}[8]                   
{\begin{CD}
#1 @>#2>> #3\\
@V{#4}VV @V{#5}VV\\
#6 @>#7>> #8
\end{CD}
}

\newcommand{\xycomsquare}[8]                   
{\xymatrix
{#1 \ar[r]^{#2} \ar[d]^{#4} &
#3 \ar[d]^{#5}  \\
#6\ar[r]^{#7} &
#8
}
}

\newcommand{\II}{{\mathbb I}}
\newcommand{\IK}{{\mathbb K}}

\newcommand{\IR}{{\mathbb R}}
\newcommand{\IS}{{\mathbb S}}

\newcommand{\IZ}{{\mathbb Z}}

\newcommand{\bfp}{{\mathbf p}}

\newcommand{\bfv}{{\mathbf v}}

\newcommand{\Bott}{\operatorname{Bott}}

\newcommand{\rk}{\operatorname{rk}}

\newcommand{\Tor}{\operatorname{Tor}}

\newcommand{\pt}{\{\bullet\}}

\newcommand{\GL}{\operatorname{GL}}

\newcommand{\higherlim}[3]{{\setbox1=\hbox{\rm lim}
        \setbox2=\hbox to \wd1{\leftarrowfill} \ht2=0pt \dp2=-1pt
        \mathop{\vtop{\baselineskip=5pt\box1\box2}}
        _{#1}}^{#2}#3}

\newcommand{\version}[1]                       
{\begin{center} last edited on #1\\
last compiled on \today \\
name of texfile: \jobname
\end{center}
}

\newcounter{commentcounter}

\begin{document}

  \begin{abstract}
In this note we present a complete computation of the topological K-theory of the reduced C*-algebra of a semidirect product of the form $\Gamma=\mathbb{Z}^n\rtimes_\rho\mathbb{Z}/2$  with no further assumptions about of the conjugacy action $\rho$. For this, we use some results for $\mathbb{Z}/2$-equivariant K-theory proved by Rosenberg and previous results of Davis and Luck when the conjugacy action $\rho$ is free outside the origin.
  \end{abstract}
  \maketitle
  \section{Introduction}


Let $\Gamma$ be a crystallographic group with holonomy $\IZ/2$, it means that $\Gamma$ is defined by an extension\begin{equation}
\label{ext}
0\to \IZ^n\to \Gamma\to \IZ/2\to 0.\end{equation}
We denote by $C_r^*(\Gamma)$ to the reduced C*-algebra of $\Gamma$. In \cite{david-luck-cr} the topological K-theory of $C_r^*(\Gamma)$ is computed, under the assumption of free conjugacy action on $\IZ^n-\{0\}$. In this note, we give a complete computation of $K_*(C_r^*(\Gamma))$ avoiding the above assumption. We use the Baum-Connes conjecture for $\Gamma$, reducing the problem, to the computation of the $\Gamma$-equivariant K-homology of the classifying space for proper actions $\underbar{E}\Gamma$. For details on the Baum-Connes conjecture consult \cite{valette} and for a proof of the Baum-Conjecture for $\Gamma$ see \cite{HK}.

The computations are obtained following a very simple idea, firstly,  $\Gamma$ can be decomposed as a pullback $\Gamma_1\times_{\IZ/2}\Gamma_2$ over $\IZ/2$, where $\Gamma_1$ has trivial conjugacy action and $\Gamma_2$ has free conjugacy action outside the origin. A model for $\underbar{E}\Gamma$ is $\IR^n$ with the natural $\Gamma$-action, as $\IZ^n$ acts freely we have a natural isomorphism
$$K_*^\Gamma(\underbar{E}\Gamma)\cong K_*^{\IZ/2}((\IS^1)^n).$$
The pullback decomposition of $\Gamma$ gives a decomposition of $(\IS^1)^n$ as a cartesian product $(\IS^1)^r\times(\IS^1)^{n-r}$, where $\IZ/2$ acts trivially on the first component and non-trivially on the second.

We start computing $\IZ/2$-equivariant K-theory of $(\IS^1)^n$ as is defined in \cite{segal}, we apply a Kunneth formula for $\IZ/2$-equivariant K-theory proved in \cite{phillips} and some results in \cite{RosKu} to the decomposition above, later we use a Universal Coefficient Theorem for equivariant K-theory proved in \cite{joachim-luck} to obtain the computation of $K_*^{\Gamma}(\underbar{E}\Gamma)$.

A similar procedure to compute the topological K-theory of crystallographic groups with others holonomy groups is not possible at the moment, because there is no generalizations of results in \cite{RosKu} for finite groups with order greater than 2, mainly because the irreducible real representations of $\IZ/n$, ($n>2$) are actually complex.
\section{Kunneth Theorem for $\IZ/2$-equivariant K-theory}
Throughout this note, K-theory or equivariant K-theory means complex topological K-theory with compact supports for locally compact
Hausdorff spaces. Bott periodicity implies that we will regard this theory as being $\IZ/2$-graded. 

Let $R=R(\IZ/2)$ be the representation ring of $\IZ/2$, which is isomorphic to $\IZ[t]/(t^2-1)$, with $t$ representing the 1-dimensional sign representation, this ring is the coefficients for $\IZ/2$-equivariant K-theory, let $I=(t-1)$ be the augmentation ideal and $J=(t+1)$, each prime ideal $\bfp$ of $R$ contains either $I$ or $J$, and these are the unique minimal prime ideals of $R$. We denote by $R_\bfp$ the localization of $R$ at the prime ideal $\bfp$.
  
The group $\IZ/2$ has exactly two irreducible real representations, the trivial and the sing, denoted by $\IR$ and by $\IR_-$ respectively. From the sign representation we can define some kind of equivariant \emph{twisted} K-theory groups $$K_{\IZ/2,-}^*(X)=K_{\IZ/2}^*(X\times \IR_-).$$ The coefficients of this theory are $K_{\IZ/2.-}^*(\pt)\cong R/J\cong I$ concentrated in even degrees.

Let us consider the inclusion $\{0\}\to \IR_-$, it induces for every $\IZ/2$-space $X$ a homomorphism of $R$-modules $$\varphi:K_{\IZ/2,-}^*(X)\to K_{\IZ/2}^*(X),$$in the other hand, consider the following composition:
$$K_{\IZ/2}^*(X)\xrightarrow{\Bott}K_{\IZ/2}^*(X\times\IR_-\times\IR_-)\to K_{\IZ/2}^*(X\times\IR_-\times\{0\}).$$It defines a homomorphism of $R$-modules
$$\psi:K_{\IZ/2}^*(X)\to K_{\IZ/2,-}^*(X).$$

\begin{proposition}\label{bbK}
	For any $\IZ/2$-space $X$ we have a natural diagram
	$$\IK_{\IZ/2}^*(X):\xymatrix{K_{\IZ/2}^*(X) \ar@/^/[r]^\psi & K_{\IZ/2,-}^*(X)\ar@/^/[l]^\varphi.}$$Where the maps $\varphi$ and $\psi$ preserves the $\IZ/2$-grading and the composite in any order is giving by multiplication by $(1-t)$. Moreover if $\bfp\subseteq R$ is a prime ideal containing $I$, then $\psi$ and $\varphi$ vanish after localizing at $\bfp$.
\end{proposition}
\begin{proof}\cite{RosKu}.
	\end{proof}

We have the following version of the Kunneth formula for $\IZ/2$-equivariant K-theory, for a proof see \cite{phillips}
\begin{theorem}\label{Kunneth}
	Let $X$ and $Y$ be $\IZ/2$-spaces, let $\bfp\subseteq R$ be a prime ideal containing $J$, with $\bfp\neq (J,2)$ there is short exact sequence of $\IZ/2$-graded, $R_\bfp$-modules  $$0\to K_G^n(X)_\bfp\otimes_{R_\bfp}K_G^m(Y)_\bfp\xrightarrow{\omega_\bfp}K_G^{m+n}(X\times Y)_\bfp\to\Tor^{R_\bfp}_1(K^n_G(X)_\bfp,K^{n+1}_G(Y)_\bfp)\to 0.$$
\end{theorem}

\begin{remark}
	When $\bfp$ contains $I$ we have a Kunneth formula taking $\IK_{\IZ/2}^*(-)$ instead of $K_{\IZ/2}^*(-)$ on Thm. \ref{Kunneth}, this is a result in \cite{RosKu}. On the other hand, as is observed in \cite{RosKu}, when $\bfp$ contains $I$, Thm. \ref{Kunneth} is not true for $K_{\IZ/2}^*(-)$ (as can be observed taking $X=Y=\IZ/2$ with the transitive action), then it is necessary to consider $\IK_{\IZ/2}^*(-)$.
\end{remark}
\section{Crystallographic groups}
Let $\Gamma$ be a group defined by the extension \ref{ext},
with the conjugation action of $\IZ/2$ given by a homomorphism $\rho:\IZ/2\to \GL(n,\IZ).$ From now on we will suppose that the conjugacy action is \emph{not} free outside the origin.

Let $H$ be the subgroup of $\IZ^n$ where $\IZ/2$ acts trivially, then there is $r\geq 1$ and a base $\{\bfv_1,\ldots,\bfv_n\}$ of $\IZ^n$ such that $\{\bfv_1,\cdots\bfv_r\}$ is a base of $H$, in this base, the homomorphism $\rho:\IZ/2\to GL(n,\IZ)$ is determined by a matrix (the image of the generator of $\IZ/2$) with the following form:

\[
\left(\begin{array}{@{}c|c@{}}

\II_r

& 0 \\
\hline
0 &

A_{n-r}

\end{array}\right)
.\]
Where the conjugation action on the last $n-r$ coordinates (denoted by $\rho_{n-r}:\IZ/2\to GL(n-r,\IZ)$) is free outside the origin.

Let us consider the canonical action of $\Gamma$ over $\IR^n$, as $\IZ^n$ acts freely, we have 
$$K_\Gamma^*(\IR^n)\cong K_{\IZ/2}^*((\IS^1)^n).$$

Now we will use Theorem \ref{Kunneth}, considering $X=(\IS^1)^r$ with the trivial $\IZ/2$-action and $Y=(\IS^1)^{n-r}$ with the $\IZ/2$-action determined by $\rho_{n-r}$. Note that $K_{\IZ/2}^*(X)$ can be computed easily (being $X$ a trivial $\IZ/2$-space) and $K_{\IZ/2}^*(Y)$ was computed in \cite{david-luck-cr}.

As the $\IZ/2$-action on $X$ is trivial, we have an isomorphism of $R$-modules $K_{\IZ/2}^*(X)\cong R\otimes_{\IZ}K^*(X)$. On the other hand $$K^*(X)\cong\begin{cases}\IZ^{2^{r-1}}& *=0\\\IZ^{2^{r-1}}&*=1.\end{cases}$$
 Then we have an isomorphism of $R$-modules

$$K^*_{\IZ/2}(X)\cong\begin{cases}R^{2^{r-1}}& *=0\\R^{2^{r-1}}&*=1.\end{cases}$$
We need to recall the following result from \cite{RosKu}.
\begin{proposition}\label{6term}
	Let $X$ be a locally
	compact $\IZ/2$-space. Then there is a natural 6-term exact sequence
	$$\xymatrix{K^1(X)\ar[r]&K^0_{\IZ/2,-}(X)\ar[r]^{\varphi}&K^0_{\IZ/2}(X)\ar[d]^{f}\\K^1_{\IZ/2}(X)\ar[u]^{f}&K^1_{\IZ/2,-}(X)\ar[l]_{\varphi}&K^0(X)\ar[l]}$$
	where the vertical arrows denoted by f on the left and right are the forgetful maps from
	equivariant to non-equivariant K-theory. 
\end{proposition}


Now we have to determine the $R$-module structure of  $K_{\IZ/2}^*(Y)$, first note that $$K^*(Y)\cong\begin{cases}\IZ^{2^{n-r-1}}& *=0\\\IZ^{2^{n-r-1}}&*=1.\end{cases}$$Now we need to recall Thm. 7.1 in \cite{david-luck-cr}.
\begin{proposition}
	\begin{enumerate}
		\item There is a split short exact sequence of $R$-modules 
		$$0\to K^0(Y/(\IZ/2))\to K^0_{\IZ/2}(Y)\to I^{2^{n-r}}\to 0$$
		\item $K^1_{\IZ/2}(Y)=0$. 
	\end{enumerate}	Where the $R$-module structure in $K^0(Y/(\IZ/2))$ is determined by the augmentation map  $R\xrightarrow{\epsilon} \IZ$.
\end{proposition}
\begin{proof}It is proved in Thm. 7.1 in \cite{david-luck-cr}, we only have to remark that the map $$K^0(Y/(\IZ/2))\to K_{\IZ/2}^0(Y)$$is the pullback of the quotient $Y\to Y/(\IZ/2)$ and is a homomorphism of $R$-modules, in a similar way the map $K^0_{\IZ/2}(Y)\to I^{2^{n-r}}$ is induced by the inclusion of representatives of the conjugacy classes of finite non-trivial subgroups of $\Gamma$, then it is a homomorphism of $R$-modules, then we have a short exact sequence of $R$-modules.
	
A splitting can be defined identifying the $R$-module $K^0(Y/(\IZ/2))$ with $K^0(Y)^{\IZ/2}$ (it happens because both groups are torsion free as abelian groups), by the forgetful map
$$K_{\IZ/2}^0(Y)\to K^0(Y)^{\IZ/2}.$$
	\end{proof}


 Identifying the $R$-modules $\IZ$ with $J$ (it can be done because $t\in R$ acts trivially) and applying the above proposition, we obtain an isomorphism of $R$-modules $$K^0_{\IZ/2}(Y)\cong J^{2^{n-r-1}}\oplus I^{2^{n-r}}.$$ 
Now we will use Thm. \ref{Kunneth} to compute $K_{\IZ/2}^*(X\times Y)_\bfp$ for every prime ideal $\bfp\subseteq R$ with $\bfp\supseteq J$ and $\bfp\neq (I,2)$, in such case we have a short exact sequence $$0\to K_{\IZ/2}^*(X)_\bfp\otimes_{R_\bfp} K_{\IZ/2}^*(Y)_\bfp\to K_{\IZ/2}^*(X\times Y)_{\bfp}\to \Tor^1_{R_\bfp}(K^*_{\IZ/2}(X)_{\bfp},K_{\IZ/2}^{*+1}(Y)_{\bfp})\to0.$$In this specific case as $K^*_{\IZ/2}(X)_{\bfp}$ is a free $R_\bfp$-module, we obtain that $$\Tor^1_{R_\bfp}(K^*_{\IZ/2}(X)_{\bfp},K_{\IZ/2}^{*+1}(Y)_{\bfp})=0,$$ then we have $$K_{\IZ/2}^*(X\times Y)_\bfp\cong K_{\IZ/2}^*(X)_\bfp\otimes_{R_\bfp} K_{\IZ/2}^*(Y)_\bfp\cong\begin{cases}
 (I_\bfp)^{2^{n-1}}&*=0\\(I_\bfp)^{2^{n-1}}&*=1.\end{cases}$$In particular $K_{\IZ/2}^*(X\times Y)_\bfp$ is torsion free as abelian group.
Now suppose $\bfp\supseteq I$, combining Prop. \ref{6term} and Prop. \ref{bbK} applied to $X\times Y$ we obtain short exact sequences 

$$0\to K^1_{\IZ/2}(X\times Y)_\bfp\xrightarrow{f} K^1(X\times Y)_\bfp\to K^0_{\IZ/2,-}(X\times Y)_\bfp\to 0$$ $$0\to K^0_{\IZ/2}(X\times Y)_\bfp\xrightarrow{f} K^0(X\times Y)_\bfp\to K^1_{\IZ/2,-}(X\times Y)_\bfp\to 0,$$where we are considering $K^*(X\times Y)$ as a $R$-module via the augmentation map $\epsilon: R\to \IZ$, in particular, $K^*_{\IZ/2}(X\times Y)_\bfp$ can be considered as a submodule of $K^*(X\times Y)_\bfp$. 

On the other hand, as $X\times Y$ is the $n$-torus we have an isomorphism of $R_\bfp$-modules
$$K^*(X\times Y)_\bfp\cong\begin{cases}J_\bfp^{2^{n-1}}&*=0\\J_\bfp^{2^{n-1}}&*=1.\end{cases}$$Then $K^*_{\IZ/2}(X\times Y)_\bfp\subseteq K^*(X\times Y)_\bfp$ is torsion free as abelian group.

As our final goal is to obtain the structure as abelian group we only need to prove that the above information implies that $K_{\IZ/2}^*(X\times Y)$ is torsion free (as abelian group), it can be done in the following way.

Let $M$ be a $R$-module, consider the \emph{$\IZ$-torsion module of $M$} defined as $$\IZ T(M)=\{x\in M\mid \text{there is } n\in \IZ-\{0\}, n\cdot x=0\}.$$Note that $\IZ T(M)$ is a $R$-submodule of $M$, and moreover $\IZ T(M)_\bfp$ can be considered as a submodule of $\IZ T(M_\bfp),$ but the above computations imply that $$\IZ T(K_{\IZ/2}^*(X\times Y)_\bfp)=0$$ for every prime ideal $\bfp\subseteq R$, then we have $$\IZ T(K_{\IZ/2}^*(X\times Y))=0,$$ and then $K_{\IZ/2}^*(X\times Y)$ is a torsion free $\IZ/2$-graded abelian group.

We only need to compute the rank of $K_{\IZ/2}^*(X\times Y)$, it can be done by the well know formula proved for example \cite{atiyah-segal-chern} or \cite{luck-chern-coh}
\begin{align*}	
\rk(K_{\IZ/2}^*(X\times Y))&=\sum_{g\in \IZ/2}\rk(K^*(X^g\times Y^g)^{C_G(g)})\\&=\rk(K^*(X\times Y)^{\IZ/2})+\rk(K^*(X^{\IZ/2}\times Y^{\IZ/2}))\\&=\rk(K^*(X))\rk(K_{\IZ/2}^*(Y))\\&=\begin{cases}3\cdot2^{n-2}&*=0\\ 3\cdot2^{n-2}&*=1.\end{cases}\end{align*}
Then we obtain.
\begin{theorem}
Let $\Gamma$ a group defined by an extension \ref*{ext}, where the action of $\IZ/2$ is \emph{not} free outside the origin, then we have an isomorphism of abelian groups

$$K_{\Gamma}^*(\underbar{E}\Gamma)\cong\begin{cases}\IZ^{3\cdot2^{n-2}}&*=0\\\IZ^{3\cdot2^{n-2}}&*=1.\end{cases}$$

\end{theorem}
\section{Topological K-theory of the reduced group C*-algebra}Now we can compute $\Gamma$-equivariant K-homology groups of $\underbar{E}\Gamma$ as is defined for example in \cite{joachim-luck}.

As $K^*_{\Gamma}(\underbar{E}\Gamma)$ is torsion free, the universal coefficient theorem for equivariant K-theory (Thm. 0.3 in \cite{joachim-luck}) reduces to $$K_*^{\Gamma}(\underbar{E}\Gamma)\cong\hom_\IZ(K^*_{\Gamma}(\underbar{E}\Gamma),\IZ)\cong\begin{cases}\IZ^{3\cdot2^{n-2}}&*=0\\ \IZ^{3\cdot 2^{n-2}}&*=1.\end{cases}$$

Finally by the Baum-Connes conjecture and previous results in \cite{david-luck-cr} we obtain a complete computation of the reduced group C*-algebra of the group $\IZ^n\rtimes\IZ/2$. 
\begin{theorem}Let $\Gamma$ a group defined by an extension \ref{ext}. If the action of $\IZ/2$ is \emph{not} free outside the origin, then we have an isomorphism of abelian groups
 $$K_*(C_r^*(\IZ^n\rtimes\IZ/2))\cong\begin{cases}\IZ^{3\cdot2^{n-2}}&*=0\\ \IZ^{3\cdot2^{n-2}}&*=1,\end{cases}$$ if the action of $\IZ/2$ is free outside the origin
$$K_*(C_r^*(\IZ^n\rtimes\IZ/2))\cong\begin{cases}\IZ^{3.2^{n-1}}&*=0\\0&*=1.\end{cases}$$
\end{theorem}

\bibliographystyle{alpha}
\bibliography{c2}
\end{document}